\documentclass[12pt]{amsart}
\usepackage{mathrsfs,amsmath,amssymb,amsfonts,amsthm,latexsym,marginnote}
\theoremstyle{plain}
\newtheorem{theorem}{Theorem}[section]
\newtheorem{proposition}[theorem]{Proposition}
\newtheorem{corollary}[theorem]{Corollary}
\newtheorem{lemma}[theorem]{Lemma}

\theoremstyle{definition}
\newtheorem{remark}[theorem]{Remark}
\newtheorem{example}[theorem]{Example}

\newcommand{\abs}[1]{\lvert#1\rvert}
\newcommand{\norm}[1]{\lVert#1\rVert}
\newcommand{\bigabs}[1]{\bigl\lvert#1\bigr\rvert}
\newcommand{\bignorm}[1]{\bigl\lVert#1\bigr\rVert}

\newcommand{\term}[1]{{\textit{\textbf{#1}}}}

\title[Unbounded order convergence]{Unbounded Order Convergence and Application to Martingales without Probability}
\date{\today}
\author[N.~Gao]{Niushan Gao}
\author[F.~Xanthos]{Foivos Xanthos*}
\address{Department of Mathematical and Statistical Sciences, University of Alberta, Edmonton, AB, Canada T6G\,2G1}
\email{niushan@ualberta.ca,foivos@ualberta.ca}
\thanks{The second author was supported by NSERC}

\begin{document}
\begin{abstract}
A net $(x_\alpha)_{\alpha\in \Gamma}$ in a vector lattice $X$ is unbounded order convergent (uo-convergent) to $x$ if $|x_\alpha-x| \wedge  y \xrightarrow{{o}} 0$ for each $y \in X_+$, and is unbounded order Cauchy (uo-Cauchy) if the net $(x_\alpha-x_{\alpha'})_{\Gamma\times \Gamma}$ is uo-convergent to $0$.
In the first part of this article, we study uo-convergent and uo-Cauchy nets in Banach lattices and
use them to characterize Banach lattices with the positive Schur property and KB-spaces.
In the second part, we use the concept of uo-Cauchy sequences to extend Doob's submartingale convergence theorems to a measure-free setting. Our results imply, in particular, that every norm bounded submartingale in $L_1(\Omega;F)$ is almost surely uo-Cauchy in $F$, where $F$ is an order continuous Banach lattice with a weak unit.\end{abstract}

\maketitle

\section{Introduction}
Recall that a net $(x_\alpha)_{\alpha\in\Gamma}$ in a vector lattice $X$ is \term{order convergent} (or, o-convergent for short) to $x\in X$ if there exists another net $(z_\beta)_{\beta\in \Lambda}$ satisfying: (1) $z_\beta\downarrow 0$; (2) for any $\beta\in\Lambda$, there exists $\alpha_0\in \Gamma$ such that $\abs{x_\alpha-x}\leq z_\beta$ for all $\alpha\geq \alpha_0$. In this case, we write $x_\alpha\xrightarrow{o}x$. In \cite{dem64,kap,nak}, the following notion was introduced and studied. In a vector lattice $X$, a net $(x_\alpha)$ is \term{unbounded order convergent} (or, uo-convergent for short) to $x\in X$ if $\abs{x_\alpha-x}\wedge y\xrightarrow{o}0$ for all $y\in X_+$. In this case, we write $x_\alpha\xrightarrow{uo}x$.
This is an analogue of pointwise convergence in function spaces. Indeed, it is easily seen that, in $c_0$ or $\ell_p (1\leq p\leq \infty)$, uo-convergence of nets is the same as coordinate-wise convergence, and that, in $L_p(\mu) (1\leq p<\infty)$, uo-convergence of sequences is the same as almost everywhere convergence.\par

Of particular interest is to study geometric properties of Banach lattices by means of uo-convergence. In \cite{wi77},  Wickstead characterized the spaces in which weak convergence of nets implies uo-convergence and vice versa. Motivated by this, we study nets which have weak and uo-convergence properties simultaneously. In Section 3, we prove that, in an order continuous Banach lattice, a relatively weakly compact uo-convergent net is absolutely weakly convergent (Proposition~\ref{abs-con}). We also characterize the spaces in which every relatively weakly compact uo-convergent net is norm convergent to be the Banach lattices with the positive Schur property (Theorem ~\ref{psp}). Moreover, we use uo-convergence as a tool to give an alternative proof of the known result that, for an order continuous Banach lattice,
the Dunford-Pettis theorem holds (i.e.~ every relatively weakly compact subset is almost order bounded) iff the space has the positive Schur property (Theorem ~\ref{spsp}).
In Section 4, we introduce the notion of uo-Cauchy nets which arises naturally from the study of martingales without probability in Section 5. An important question is to search for conditions under which a uo-Cauchy net is uo-convergent. We prove that, in an order continuous Banach lattice, every relative weakly compact uo-Cauchy net is uo-convergent (Theorem~\ref{wcuc}). We also prove that, in KB-spaces, every norm bounded uo-Cauchy net is uo-convergent. This property actually characterizes KB-spaces among order continuous Banach lattices (Theorem ~\ref{cha}).\par

The second part of our study, Section 5, is devoted to an application of uo-convergence to the theory of martingales without probability. In this theory, martingales are defined in vector and Banach lattices, with respect to a sequence of positive projections that play the role of filtration. To our knowledge, this theory was initiated by DeMarr~\cite{dem66}, in which it is proved that, under certain natural conditions on the filtration, order bounded submartingales are order convergent. However, as stated in \cite{dem66}, their main theorem does not include Doob's almost sure convergence theorem as a special case. Since then, other efforts have been made to this problem. See, for example, \cite{tro11,koro08,kuo04,kuo05,kuo06,sto,tro05,uhl71}.
In particular, in \cite{kuo05}  a new type of natural conditions on the filtration was introduced. It is under these conditions that we establish a version of Doob's almost sure convergence theorem for vector lattices in terms of uo-Cauchy sequences (Theorem ~\ref{pcon1}). In particular, we prove that, if the space is a KB-space then every normed bounded submartingale is uo-convergent (Proposition ~\ref{nconf}). These results include the classical Doob's theorem as a special case. In addition, we establish some results on norm convergence of submartingales in Banach lattices (Theorems~\ref{partial} and~\ref{acon1}). It deserves mentioning that our results apply to Banach lattice valued submartingales. Precisely, we prove that every norm bounded submartingale in $L_1(\Omega;F)$ with respect to classical filtrations is almost surely uo-Cauchy in $F$, where $F$ is an order continuous Banach lattice with a weak unit (Corollary ~\ref{vmart}) .

\section{Preliminaries}

\subsection{Some basic definitions}We adopt \cite{abr,ali,mn1991} as standard references for basic notions on Banach spaces, Banach lattices and vector lattices. Recall that a vector lattice is said to be \term{order complete} (respectively, \term{$\sigma$-order complete}) if every order bounded above subset (respectively, countable subset) has a supremum.
It is easily seen that an order bounded net $(x_\alpha)$ in an order complete vector lattice $X$ order converges to $x\in X$ if and only if $\inf_\alpha \sup_{\beta\geq \alpha}\abs{x_\beta-x}=0$.
Recall also that in a vector lattice, a positive vector $x_0>0$ is called a \term{weak unit} if $x\wedge nx_0\uparrow x $ for all $x\geq 0$, and that in a normed lattice, a positive vector $x_0>0$ is called a \term{quasi-interior point} if $x\wedge nx_0\rightarrow x$ in norm for all $x\geq 0$. Quasi-interior points are weak units. A positive functional on a vector lattice is said to be \term{strictly positive} if it does not vanish on non-zero positive vectors, and is said to be \term{order continuous} if it maps order null nets to order null nets.\par

A Banach lattice is said to be \term{order continuous} or \term{have order continuous norm} if every order null net is norm null, is said to be a \term{KB}-space if every norm bounded increasing sequence converges in norm, and is said to be an \term{AL}-space if $\norm{x+y}=\norm{x}+\norm{y}$ for all $x,y\geq0$. It is well-known that KB-spaces are order continuous and that AL-spaces are KB-spaces. We also need the following facts. In order continuous spaces, weak units are quasi-interior points. In KB-spaces, norm bounded increasing nets are norm convergent. Moreover, a Banach lattice is a KB-space iff it contains no lattice copies of $c_0$ iff it is a band of its second dual (cf.~\cite[Theorem~4.60]{ali}). Kakutani's representation theorem (cf.~\cite[Theorem~4.27]{ali}) says that every AL-space is lattice isometric to $L_1(\mu)$ for some measure $\mu$. If, in addition, the AL-space has a weak unit $x_0$, then $\mu$ can be chosen finite with $x_0$ corresponding to the constant one function (cf.~the proof of \cite[Theorem~4.27]{ali}).\par

In a Banach lattice $X$, a subset $A$ is said to be \term{almost order bounded} if for any $\varepsilon>0$ there exists $u\in X_+$ such that $A\subset [-u,u]+\varepsilon B_X$. One should observe the following useful fact, which can be easily verified using Riesz decomposition theorem, that $A\subset [-u,u]+\varepsilon B_X$ iff
$\sup_{x\in A}\bignorm{(\abs{x}-u)^+}= \sup_{x\in A}\bignorm{\abs{x}-u\wedge \abs{x}}\leq \varepsilon$. Relatively compact subsets (in particular, norm convergent sequences) are easily seen to be almost order bounded. Almost order bounded subsets in order continuous Banach lattices are relatively weakly compact (cf.~\cite[Theorem~4.9(5) and Theorem 3.44]{ali}). The classical Dunford-Pettis theorem asserts that in $L_1(\mu)$, a subset is relatively weakly compact iff it is almost order bounded (apply e.g.~\cite[Theorem~4.37]{ali} with the functional being integration), iff it is norm bounded and uniformly integrable when $\mu$ is finite (cf.~\cite[Theorem~5.2.9]{kalton}).

\subsection{A representation theorem}\label{AL-rep} A fundamental tool used in this paper is the following standard representation theorem of vector lattices. We refer the reader to \cite[Proposition~2.4.16]{mn1991} for specific details.\par

\emph{Throughout this subsection, let $X$ stand for an order complete vector lattice with a strictly positive order continuous functional $x_0^*>0$}.
It is clear that the norm completion $\widetilde{X}$ of $X$ with respect to the norm
$$\bignorm{x}_L=x_0^*(|x|),\qquad \forall~x\in X.$$
is an AL-space. By \cite[Proposition~2.4.16]{mn1991}, $X$ is an ideal of $\widetilde{X}$. It is easily seen that $X$ is norm dense, and thus is order dense, in $\widetilde{X}$.

\begin{lemma}\label{representation}Let $x_0$ be a weak unit of $X$. Then $\widetilde{X}$ is lattice isometric to $L_1(\mu)$ for some finite measure $\mu$ with $x_0$ corresponding to the constant $1$ function.
\end{lemma}

\begin{proof}In view of Kakutani's representation theorem, it suffices to show that $x_0$ is also a weak unit of $\widetilde{X}$. Indeed, the band generated by $x_0$ in $\widetilde{X}$ clearly contains $X$, and hence is equal to $\widetilde{X}$.
\end{proof}

It is critical for us to know how uo-convergence in $X$ passes to $\widetilde{X}$ and also how weak convergence passes when $X$ is a normed lattice. The following is immediate by Lemma~\ref{uo-ideal} (see below).

\begin{lemma}\label{uo-equiv}For $(x_\alpha)\subset X$ and $x\in X$, $x_\alpha\xrightarrow{uo}x$ in $X$ iff
$x_\alpha\xrightarrow{uo}x$ in $\widetilde{X}$.
\end{lemma}
It follows that, for a sequence $(x_n)\subset X$ and $x\in X$, $x_n\xrightarrow{uo}x$ in $X$ iff $x_n\xrightarrow{uo}x$ in $\widetilde{X}$ iff $x_n\rightarrow x$ almost everywhere in $\widetilde{X}$.\par

If $X$ is normed and $x_0^*$ is norm continuous then $X$ embeds continuously into $\widetilde{X}$. Thus, the following is immediate.

\begin{lemma}\label{al-wcon}Let $X$ be a normed lattice and $x_0^*$ be norm continuous.
\begin{enumerate}
\item For $(x_\alpha)\subset X$ and $x\in X$, if $x_\alpha\xrightarrow{w}x$ in $X$, then
$x_\alpha\xrightarrow{w}x$ in $\widetilde{X}$.
\item If $A\subset X$ is weakly compact in $X$, then $A$ is also weakly compact in $\widetilde{X}$.
\end{enumerate}
\end{lemma}

We end this section with a simple result on extensions of positive operators.

\begin{lemma}\label{ope} For any positive operator $T$ on $X$, if $T^*x_0^*=x_0^*$, then $T$ extends uniquely to a contractive positive operator on $\widetilde{X}$.\end{lemma}

\begin{proof}
Indeed, for any $x\in X$, $\bignorm{Tx}_L=x_0^*(|Tx|)\leq x_0^*(T|x|)=x_0^*(|x|)=\bignorm{x}_L$. Hence, $T$ extends uniquely to a contractive positive operator $\widetilde{X}$.
\end{proof}

\section{Unbounded Order Convergence}\label{uo-convergence}
\subsection{Some basic properties}Note that the uo-limit is unique whenever it exists. Note also that, for order bounded nets, uo-convergence is equivalent to o-convergence. We now provide some basic properties of uo-convergent nets that will be used in the sequel. We include the proofs of the first two lemmas for the sake of completeness.

\begin{lemma}[\cite{kap}]\label{comparison}\begin{enumerate}
\item\label{comp0} Suppose $x_\alpha\xrightarrow{uo} x$ and $y_\alpha\xrightarrow{uo} y$. Then $a x_\alpha+b y_\alpha\xrightarrow{uo}a x +by$ for any $a,b\in\mathbb{R}$.
\item\label{comp1} $x_\alpha\xrightarrow{uo}x$ iff $x_\alpha^\pm\xrightarrow{uo}x^\pm$. In this case, $\abs{x_\alpha}\xrightarrow{uo}\abs{x}$.
\item\label{comp2} Suppose $0\leq x_\alpha\xrightarrow{uo} x$ and $x_\alpha\leq y_\alpha\xrightarrow{uo} y$. Then $0\leq x\leq y$.
\end{enumerate}
\end{lemma}

\begin{proof}\eqref{comp0} is trivial. \eqref{comp1} follows from $\abs{x^\pm-y^\pm}\wedge \abs{z}\leq \abs{x-y}\wedge \abs{z}\leq \abs{x^+-y^+}\wedge\abs{z}+\abs{x^--y^-}\wedge \abs{z}$ and $\bigabs{\abs{x}-\abs{y}}\leq \abs{x-y}$. For \eqref{comp2}, note that $x_\alpha=\abs{x_\alpha}\xrightarrow{uo}\abs{x}$ by \eqref{comp1}. Thus, $x=\abs{x}\geq 0$, by uniqueness of uo-limits. Since $0\leq y_\alpha-x_\alpha\xrightarrow{uo}y-x$, we have $y\geq x$.
\end{proof}

\begin{lemma}[\cite{kap}]\label{uoc}Let $X$ be an order complete vector lattice with a weak unit $x_0$. Then $x_\alpha\xrightarrow{uo}0$ iff $|x_\alpha|\wedge x_0\xrightarrow{o}0$.\end{lemma}

\begin{proof}The ``only if'' part is obvious. We prove the ``if'' part. Pick any $y\geq 0$. Since $X$ is order complete, we have that
$$\left(\inf_\alpha \sup_{\beta\geq \alpha}\left(\abs{x_\beta}\wedge y\right)\right)\wedge x_0=\left(\inf_\alpha \sup_{\beta\geq \alpha} \left(\abs{x_\beta}\wedge x_0\right)\right)\wedge y=0\wedge y=0.$$
Thus, $x_0$ being a weak unit implies $\inf_\alpha \sup_{\beta\geq \alpha}\left(\abs{x_\beta}\wedge y\right)=0$, and $\abs{x_\alpha}\wedge y\xrightarrow{o}0$.
\end{proof}

For the next two lemmas, observe the following fact. Let $X$ be a vector lattice, $I$ an ideal of $X$ and $(x_\alpha)\subset I$. If $x_\alpha\xrightarrow{o}0$ in $I$, then $x_\alpha\xrightarrow{o}0$ in $X$. Conversely, if $(x_\alpha)$ is order bounded in $I$ and $x_\alpha\xrightarrow{o}0$ in $X$, then $x_\alpha\xrightarrow{o}0$ in $I$.

\begin{lemma}\label{band-proj}
Let $B$ be a projection band and $P$ the corresponding band projection. If $x_\alpha\xrightarrow{uo}x$ in $X$ then $Px_\alpha\xrightarrow{uo}Px$ in both $X$ and $B$.
\end{lemma}
\begin{proof}Note that $P$ is a lattice homomorphism and $0\leq P\leq I$. Since $\abs{Px_\alpha-Px}=P\abs{x_\alpha-x}\leq \abs{x_\alpha-x}$, it is easily seen that $Px_\alpha\xrightarrow{uo}Px$ in $X$. In particular, for any $y\in B_+$, $\abs{Px_\alpha-Px}\wedge y\xrightarrow{o}0$ in $X$, and thus in $B$, by the preceding remark. Hence, $P_{x_\alpha} \xrightarrow{{uo}} Px$ in $B$.
\end{proof}

\begin{lemma}\label{uo-ideal}
Let $X$ be an order complete vector lattice and $I$ an ideal of $X$. For $(x_\alpha)\subset I$, $x_\alpha\xrightarrow{uo}0$ in $I$ iff $x_\alpha\xrightarrow{uo}0$ in $X$.
\end{lemma}

\begin{proof}The ``if'' part. Suppose $x_\alpha\xrightarrow{uo}0$ in $X$. Then for any $y\in I_+$, $\abs{x_\alpha}\wedge y\xrightarrow{o}0$ in $X$. The preceding remark implies $\abs{x_\alpha}\wedge y\xrightarrow{o}0$ in $I$. Therefore, $x_\alpha\xrightarrow{uo}0$ in $I$.\par
The ``only if'' part. Suppose $x_\alpha\xrightarrow{uo}0$ in $I$.
Pick any $y\in I_+$. Then $\abs{x_\alpha}\wedge y\xrightarrow{o}0$ in $I$, and thus in $X$, by the preceding remark again.
It follows that, for any $y\in I_+$ and any $0\leq z\in I^{\rm d}$,
$\abs{x_\alpha}\wedge (y+z)=\abs{x_\alpha}\wedge y\xrightarrow{o}0$ in $X$.\par

Now for any $w\in X_+$ and any $u\in (I\oplus I^{\rm d})_+$, we have $w\wedge u\in (I\oplus I^{\rm d})_+$. Thus by the preceding paragraph, $\abs{x_\alpha}\wedge(w\wedge u)\xrightarrow{o}0$ in $X$, or equivalently, $$\left(\inf_\alpha\sup_{\beta\geq \alpha}\abs{x_\beta}\wedge w\right)\wedge u=\inf_\alpha\sup_{\beta\geq \alpha}\left(\abs{x_\beta}\wedge(w\wedge u)\right)=0.$$
Observe that $(I\oplus I^{\rm d})^{\rm d}=\{0\}$ by \cite[Theorem~1.36~(2) and (1)]{ali}. Therefore, we have, due to arbitrariness of $u$,
$$\inf_\alpha\sup_{\beta\geq \alpha}(\abs{x_\beta}\wedge w)
=0.$$
It follows that $\abs{x_\alpha}\wedge w\xrightarrow{o}0$ in $X$. Hence, $x_\alpha\xrightarrow{uo}0$ in $X$.
\end{proof}

\begin{remark}In Lemmas~\ref{uoc} and \ref{uo-ideal}, it suffices to assume $X$ is $\sigma$-order complete if we only consider nets with countable index sets (in particular, sequences).
\end{remark}

\begin{lemma}\label{norm}Let $X$ be an order continuous Banach lattice. Suppose $x_\alpha\xrightarrow{uo}x$. Then $\norm{x}\leq\liminf_\alpha\norm{x_\alpha}$.
\end{lemma}

\begin{proof}
Clearly, $\bigabs{\abs{x_\alpha}\wedge \abs{x}-\abs{x}\wedge \abs{x}}\leq \bigabs{\abs{x_\alpha}-\abs{x}}\wedge\abs{x}\leq \abs{x_\alpha-x}\wedge \abs{x}\xrightarrow{o}0$. Thus $\abs{x_\alpha}\wedge \abs{x}\xrightarrow{\norm{\cdot}} \abs{x}$.
It follows that $\norm{x}\leq\liminf_\alpha\norm{\abs{x_\alpha}\wedge\abs{x}}\leq\liminf_\alpha\norm{x_\alpha}$.
\end{proof}

\subsection{Norm convergence of uo-convergent nets}
Recall that a subset in $L_1(\mu)$ is relatively weakly compact iff it is almost order bounded (the classical Dunford-Pettis theorem), and iff it is norm bounded and uniformly integrable when $\mu$ is finite.
Recall also that, when $\mu$ is finite, a uniformly integrable, almost everywhere convergent sequence in $L_1(\mu)$ converge in norm. In view of this, the following results may be viewed as \term{Abstract Dominated Theorem}.

\begin{proposition}\label{ncon}
Let $X$ be an order continuous Banach lattice. If $(x_\alpha)$ is almost order bounded\footnote{We mean that the set of members in the net is almost order bounded.} and uo-converges to $x$, then $x_\alpha$ converges to $x$ in norm.\end{proposition}
\begin{proof}Note that the net $\{\abs{x_\alpha-x}\}$ is also almost order bounded. Fix any $\varepsilon>0$. Then there exists $u>0$ such that, for all $\alpha$,
\begin{align}\label{ineq1}\bignorm{|x_\alpha-x|-|x_\alpha-x|\wedge u}=\bignorm{(|x_\alpha-x|-u)^+}\leq \varepsilon.\end{align}
Since $|x_\alpha-x|\wedge u$ converges to $0$ in order of $X$, we know
\begin{align}\label{ineq2}\bignorm{|x_\alpha-x|\wedge u}\rightarrow  0.\end{align}
Combining \eqref{ineq1} and \eqref{ineq2}, we have $x_\alpha\rightarrow x$ in norm.
\end{proof}

We need the following lemma, which is a  special case of \cite[Theorem~4.37]{ali}. For the convenience of the reader,
we provide a short proof of it based on the classical Dunford-Pettis Theorem; it also illustrates the power of the representation technique from Subsection 2.2.

\begin{lemma}\label{aobgeneral}Let $X$ be an order continuous Banach lattice. Let $A$ be a relatively weakly compact subset of $X$. Then for any $\varepsilon>0$ and $x^* \in X^*_+$, there exists $u\in X_+$ such that $\sup_{x\in A}x^*\left((\abs{x}-u)^+\right)<\varepsilon$.
\end{lemma}

\begin{proof}Assume first that $x^*$ is strictly positive. Since $X$ is order continuous, $X$ is order complete and $x^*$ is order continuous. Let $\widetilde{X}$ be the AL-space constructed for the pair $(X,x^*)$ as in Subsection~\ref{AL-rep}.
By Lemma~\ref{al-wcon}, $A$ is relatively weakly compact in the AL-space $\widetilde{X}$, and thus is almost order bounded in $\widetilde{X}$ by Dunford-Pettis theorem. Therefore, there exists $v\in \widetilde{X}$ such that $$\sup_{x\in A}x^*\left((\abs{x}-v)^+\right)=\sup_{x\in A}\bignorm{(\abs{x}-v)^+}_L<\varepsilon.$$
Since $X$ is dense in $\widetilde{X}$, there exists $u\in X$ such that $$x^*(\abs{u-v})=\norm{u-v}_L<\varepsilon.$$ It follows from $(\abs{x}-u)^+\leq (\abs{x}-v)^++\abs{v-u}$ that $\sup_{x\in A}x^*\left((\abs{x}-u)^+\right)<2\varepsilon$.\par

The general case. Put $N_{x^*}=\{x:x^*(\abs{x})=0\}$ and $C_{x^*}=N^{\rm d}_{x^*}$. Then they are both bands of $X$ and are order continuous Banach lattices in their own right. Note that $X=N_{x^*}\oplus C_{x^*}$. Let $P$ be the band projection onto $C_{x^*}$. Clearly, $P(A)$ is relatively weakly compact in $C_{x^*}$ and $x^*$ is strictly positive on $C_{x^*}$. Thus, by the preceding case, there exists $u\in C_{x^*}$ such that $\sup_{x\in A}x^*\left((P\abs{x}-u)^+\right)<\varepsilon$. It follows from $(\abs{x}-u)^+\leq (P\abs{x}-u)^++(I-P)\abs{x}$ and $x^*((I-P)\abs{x})=0$ that $\sup_{x\in A}x^*\left((\abs{x}-u)^+\right)<\varepsilon$.
\end{proof}

\begin{proposition}\label{abs-con}Let $X$ be an order continuous Banach lattice. If $(x_\alpha)$ is relatively weakly compact and uo-converges to $x$, then $x_\alpha$ converges to $x$ in $\abs{\sigma}(X,X^*)$.
\end{proposition}

\begin{proof}Pick any $x^*\in X^*_+$ and $\varepsilon>0$. By Lemma~\ref{aobgeneral}, there exists $u\in X_+$ such that, for all $\alpha$,
\begin{align}\label{ineq3}x^*\left(\abs{x_\alpha-x}-\abs{x_\alpha-x}\wedge u\right)=x^*\left((\abs{x_\alpha-x}- u)^+\right)<\varepsilon.\end{align}
Since $\abs{x_\alpha-x}\wedge u\xrightarrow{o}0$ and $x^*$ is order continuous, we have
\begin{align}\label{ineq4}
x^*(\abs{x_\alpha-x}\wedge u)\rightarrow 0.
\end{align}Combining \eqref{ineq3} and \eqref{ineq4}, we have $x^*(\abs{x_\alpha-x})\rightarrow0$.
\end{proof}

\begin{remark}\label{follow-ncon}
Note that, in general, we cannot replace almost order boundedness by weak compactness in Proposition~\ref{ncon}, or equivalently, we can not expect norm convergence in Proposition~\ref{abs-con}. Indeed, let $(e_n)$ be the standard basis of $\ell_2$ or $c_0$, then $e_n$ weakly and uo-converges to $0$, but $\norm{e_n}=1$ for all $n$.
\end{remark}

We now characterize the spaces where we can expect norm convergence. Recall that a Banach lattice $X$ is said to have the \term{positive Schur property} if $\norm{x_n}\rightarrow0$ whenever $0\leq x_n\xrightarrow{w}0$ in $X$. It is easily seen that the standard basis of $c_0$ is positive, weakly null, but not norm null. Thus, any Banach lattice with the positive Schur property cannot contain lattice copies of $c_0$, and therefore, is a KB-space. In particular, it is order continuous.

\begin{theorem}\label{psp}Let $X$ be a $\sigma$-order complete Banach lattice. The following are equivalent:
\begin{enumerate}
\item\label{psp1} $X$ has the positive Schur property;
\item\label{psp2} for every relatively weakly compact net $(x_\alpha)\subset X$, $0\leq x_\alpha\xrightarrow{w}0$ implies $x_\alpha\xrightarrow{\norm{\cdot}}0$;
\item\label{psp3} for any sequence $(x_n)\subset X$ and $x\in X$,  $x_n\xrightarrow[uo]{w}x$ implies $x_n\xrightarrow{\norm{\cdot}}x$;
\item\label{psp4} for any relatively weakly compact net $(x_\alpha)\subset X$ and $x\in X$, $x_\alpha\xrightarrow{uo}x$ implies $x_\alpha\xrightarrow{\norm{\cdot}}x$;
\item\label{psp5} for any sequence $(x_n)\subset X$, $0\leq x_n\xrightarrow[uo]{w}0$ implies $x_n\xrightarrow{\norm{\cdot}}0$;
\item\label{psp6} for any relatively weakly compact net $(x_\alpha)$, $0\leq x_\alpha\xrightarrow{uo}0$ implies $x_\alpha\xrightarrow{\norm{\cdot}}0$;
    \end{enumerate}
\end{theorem}

\begin{proof}We first establish the equivalence of \eqref{psp1}, \eqref{psp3} and \eqref{psp5}. The implications \eqref{psp1}$\Rightarrow$\eqref{psp5} and \eqref{psp3}$\Rightarrow$\eqref{psp5} are obvious.\par

Suppose \eqref{psp5} holds. We first show that $X$ is an order continuous Banach lattice. Indeed, take any disjoint order bounded sequence $(x_n)$. Suppose $\abs{x_n}\leq x$ for all $n$. For any $x^*\in X^*_+$, we have $\sum_1^nx^*(\abs{x_i})=x^*(\sum_1^n\abs{x_i})=x^*(\vee_1^n \abs{x_i})\leq x^*(x)$ for all $n$. Thus $\sum_1^\infty x^*(\abs{x_n})$ converges. In particular, $x^*(\abs{x_n})\rightarrow0$. It follows that $\abs{x_n}\xrightarrow{w}0$. Now we claim $x_n\xrightarrow{o}0$. Since $X$ is $\sigma$-order complete, it suffices to prove $y_n:=\sup_{m\geq n}\abs{x_m}\downarrow0$. Suppose $0\leq u\leq y_n$ for all $n$. Then $u=y_n\wedge u=\left(\sup_{m\geq n}(\abs{x_m}\wedge u)\right)\perp x_{n-1}$ for all $n\geq 2$. It follows that $u=y_1\wedge u=\sup_{n\geq 1}(\abs{x_n}\wedge u)=0$. This proves the claim. Therefore, $\norm{x_n}\rightarrow0$ by the assumption \eqref{psp5}. Hence, $X$ is order continuous (cf.~\cite[Theorem~4.14]{ali}).\par

\eqref{psp5}$\Rightarrow$\eqref{psp1} Suppose $0\leq x_n\xrightarrow{w}0$ but $\norm{x_n}\not\rightarrow0$. By passing to a subsequence, we may assume $\inf_n\norm{x_n}>0$. Let $B$ be the band generated by $(x_n)$. Since $X$ is order continuous, so is $B$. Thus, $B$ having a weak unit implies that it has a strictly positive functional $x_0^*>0$ (cf.~\cite[Theorem~4.15]{ali}).
Let $\widetilde{B}$ be the AL-space constructed for the pair $(B,x_0^*)$ as in Subsection~\ref{AL-rep}. Then $\norm{x_n}_L=x_0^*(x_n)\rightarrow 0$. Therefore, $(x_n)$ has a subsequence $(x_{n_k})$ which converges in order to $0$ in $\widetilde{B}$\footnote{Suppose $\norm{x_n}\rightarrow 0$ in a Banach lattice. Take a subsequence $(x_{n_k})$ such that $\norm{x_{n_k}}\leq 2^{-k}$. Put $y_k=\sum_{j=k}^\infty \abs{x_{n_j}}$. Then $\abs{x_{n_k}}\leq y_k\downarrow 0$, and thus $x_{n_k}\xrightarrow{o}0$.}. In particular, $x_{n_k}\xrightarrow{uo}0$ in $\widetilde{B}$, hence in $B$ by Lemma~\ref{uo-equiv}, and in $X$ by Lemma~\ref{uo-ideal}. By the assumption \eqref{psp5}, we have $\norm{x_{n_k}}\rightarrow0$, a contradiction.\par

\eqref{psp5}$\Rightarrow$\eqref{psp3} Note that $X$ is order continuous. Take any $x_n\xrightarrow[uo]{w}x$. Since $(x_n)$ is relatively weakly compact, we have $\abs{x_n-x}\xrightarrow{w}0$ by Proposition~\ref{abs-con}. Thus it follows from $\abs{x_n-x}\xrightarrow{uo}0$ and the assumption \eqref{psp5} that $x_{n}\xrightarrow{\norm{\cdot}}x$.\par

\eqref{psp1}$\Rightarrow$\eqref{psp2} Suppose $(x_\alpha)$ is relatively weakly compact and $0\leq x_\alpha\xrightarrow{w}0$. Suppose also $x_\alpha$ does not converge to $0$ in norm. Then there exists $\varepsilon>0$ such that for any $\alpha$, there exists $\beta(\alpha)\geq \alpha$ satisfying $\norm{x_{\beta(\alpha)}}\geq \varepsilon$. Thus, by passing to the subnet $(x_{\beta(\alpha)})$, we may assume $\inf_\alpha\norm{x_\alpha}>0$. Since $0\in \overline{\{x_\alpha:\alpha\}}^w$, there exists a sequence $(y_n)_1^\infty\subset \{x_\alpha:\alpha\}$ such that $y_n\xrightarrow{w}0$, by \cite[Theorem~4.50]{fab2001}. The positive Schur property now implies $\norm{y_n}\rightarrow 0$, a contradiction.\par

\eqref{psp2}$\Rightarrow$\eqref{psp4} The assumption \eqref{psp2} clearly implies $X$ has the Schur property, hence $X$ is a KB-space. Now suppose $(x_\alpha)$ is relatively weakly compact and $x_\alpha\xrightarrow{uo}x$. Then $\abs{x_\alpha-x}\xrightarrow{w}0$ by Proposition~\ref{abs-con}. By \cite[Theorem~4.39]{ali}, $\{\abs{x_\alpha-x}:\alpha\}$ is also relatively weakly compact. Therefore, it follows from the assumption \eqref{psp2} that $\norm{x_\alpha-x}\rightarrow 0$.\par

The implications \eqref{psp4}$\Rightarrow$\eqref{psp6}$\Rightarrow$\eqref{psp5} are obvious.
\end{proof}

\subsection{Dunford-Pettis theorem in Banach lattices}
We end this section with a known characterization of order continuous Banach lattices in which the Dunford-Pettis theorem holds, namely, every relatively weakly compact subset is almost order bounded. It follows from \cite[Theorem 7]{wnuk} (cf.~\cite[Theorem~3.1]{chen}) and \cite[Proposition~3.6.2]{mn1991}. We give an alternative proof using uo-convergence.

\begin{lemma}\label{wc-aob}Let $X$ be a Banach lattice with the positive Schur property. Suppose that $X$ has a weak unit $x_0$. Then every relatively weakly compact subset of $X$ is almost order bounded.
\end{lemma}

\begin{proof}Let $A\subset X$ be relatively weakly compact. We claim that, $\forall\;x^*\in X^*_+$, \begin{align}\label{unif}\lim_n\sup_{x\in A}x^*\left((\abs{x}-nx_0)^+\right)=0.\end{align}
Indeed, since $X$ has the positive Schur property, it is order continuous. Thus, by Lemma~\ref{aobgeneral}, for any $x^*>0$ and any $\varepsilon>0$, there exists $u\in X_+$ such that $\sup_{x\in A}x^*\left((\abs{x}-u)^+\right)<\varepsilon$. Since $x_0$ is a weak unit of ${X}$, there exists $n_0$ such that $\norm{(u-nx_0)^+}=\norm{u-u\wedge nx_0}<\varepsilon$ for all $n\geq n_0$. It follows that $x^*\left({(u- nx_0)^+}\right)<\norm{x^*}\varepsilon$ for all $n\geq n_0$. Hence,
$\sup_{x\in A}x^*\left((\abs{x}-nx_0)^+\right)\leq (1+\norm{x^*})\varepsilon$ for all $n\geq n_0$. Therefore, $\lim_n\sup_{x\in A}x^*\left((\abs{x}-nx_0)^+\right)=0$.\par

We now show that, for any $\varepsilon>0$, there exists $n$ such that $\sup_{x\in A}\bignorm{(\abs{x}-nx_0)^+}\leq \varepsilon$. Suppose not. Then there exists $\varepsilon>0$ such that for any $n\geq 1$ there exists $x_n\in A$ with $\norm{(\abs{x_n}-nx_0)^+}>\varepsilon$. But we have, by \eqref{unif}, $(\abs{x_n}-nx_0)^+\xrightarrow{w}0$. Thus, $\bignorm{(\abs{x_n}-nx_0)^+}\rightarrow0$ by the positive Schur property, a contradiction.
\end{proof}

\begin{theorem}\label{spsp}Let $X$ be an order continuous Banach lattice. The following are equivalent:
\begin{enumerate}
\item\label{spsp1} $X$ has the positive Schur property;
\item\label{spsp2} every relatively weakly compact countable subset of $X$ is almost order bounded;
\item\label{spsp3} every relatively weakly compact subset of $X$ is almost order bounded.
\end{enumerate}
\end{theorem}

\begin{proof}\eqref{spsp1}$\Rightarrow$\eqref{spsp3}  Let $\{y_{\gamma} :\gamma\in\Gamma\}$ be a maximal collection of pairwise disjoint elements of $X$. Let $\Delta$ be the collection of all finite subsets of $\Gamma$ directed by inclusion. For each $\delta=\{\gamma_1,\dots,\gamma_n\}\in\Delta$, the band $B_\delta$ generated by $\{y_{\gamma_i}\}_1^n$ is an order continuous Banach lattice with the positive Schur property and has a weak unit $y_\delta=\sum_1^ny_{\gamma_i}$. Let $P_\delta$ be the band projection onto $B_\delta$. Observe that $P_\delta x\uparrow x$ for all $x\in X_+$. Indeed, for any $x\in X_+$, $(P_\delta x)$ is increasing and norm bounded. Since $X$ is a KB-space, $(P_\delta x)$ is norm convergent. Thus, for each $\gamma \in \Gamma$, we have $\abs{x-\lim P_\delta x} \wedge y_\gamma=\lim \big(\abs{(I-P_\delta)x}\wedge y_\gamma\big)=0$. Now maximality of $\mathscr{C}$ implies $x=\lim P_\delta x$. This proves the observation.\par

Let $A$ be relatively weakly compact.
We first show that \begin{align}\label{dont}\inf_\delta\sup_{x\in A}\bignorm{(I-P_\delta)\abs{x}}=0.\end{align}
Suppose, otherwise, $2c:=\inf_\delta\sup_{x\in A}\bignorm{(I-P_\delta)\abs{x}}>0$. Then for each $\delta$, there exists $x_\delta\in A$ such that $\bignorm{(I-P_{\delta})\abs{x_\delta}}>c$. Consider the net $\big((I-P_\delta)\abs{x_\delta}\big)$. We claim that $(I-P_\delta)\abs{x_\delta}\xrightarrow{uo}0$. Indeed, for any $y\in X_+$, we have $
\big((I-P_\delta)\abs{x_\delta}\big)\wedge y=(I-P_\delta)\Big(\big((I-P_\delta)\abs{x_\delta}\big)\wedge y\Big)
\leq (I-P_\delta)y\xrightarrow{o}0$. This proves the claim.
Note also that, as a subset of the solid hull of $A$, $\big((I-P_\delta)\abs{x_\delta}\big)$ is also relatively weakly compact by \cite[Theorem~4.39]{ali}. Therefore, $\bignorm{(I-P_\delta)\abs{x_\delta}}\rightarrow 0$ by Theorem~\ref{psp}\eqref{psp6}, a contradiction. This completes the proof of \eqref{dont}.\par

Therefore, for any $\varepsilon>0$, we can find $\delta$ such that \begin{align}\label{nnnn1}\sup_{x\in A}\bignorm{(I-P_\delta)\abs{x}}<\varepsilon.\end{align}
By \cite[Theorem~4.39]{ali} again, $P_\delta(\abs{A})$ is relatively weakly compact in $X$, and hence in $B_\delta$. Since $B_\delta$ has the positive Schur property and a weak unit, by Lemma~\ref{wc-aob}, there exists $0<u\in B_\delta$ such that \begin{align}\label{nnnn2}\sup_{x\in A}\bignorm{(P_\delta\abs{x}-u)^+}=\sup_{x\in A}\bignorm{P_\delta\abs{x}-(P_\delta\abs{x})\wedge u}<\varepsilon.\end{align}
Combining \eqref{nnnn1} and \eqref{nnnn2}, one gets $\sup_{x\in A}\bignorm{\abs{x}-\abs{x}\wedge u}\leq \sup_{x\in A}\bignorm{\abs{x}-(P_\delta\abs{x})\wedge u}<2\varepsilon$.
Hence, $A$ is almost order bounded.\par

The implication \eqref{spsp3}$\Rightarrow$\eqref{spsp2} is obvious. We now prove \eqref{spsp2}$\Rightarrow$\eqref{spsp1}. Take any sequence $(x_n)$ with $0\leq x_n\xrightarrow[uo]{w}0$. Then $(x_n)$ is relatively weakly compact, and therefore, almost order bounded by assumption. Proposition~\ref{ncon} implies $\norm{x_n}\rightarrow 0$. Hence, $X$ has the positive Schur property by Theorem~\ref{psp}\eqref{psp5}.
\end{proof}

\section{Unbounded Order Cauchy}

In view of the fact that the almost everywhere limit of a sequence in $L_1$ may not belong to $L_1$, we introduce the following notion. In a vector lattice $X$, a net $(x_\alpha)$ is said to be \term{unbounded order Cauchy} (or, uo-Cauchy for short), if the net $(x_\alpha-x_{\alpha'})_{(\alpha,\alpha')}$ uo-converges to $0$.\par

It is easily seen that, for order bounded nets, uo-Cauchy is equivalent to o-Cauchy, and therefore, in order complete lattices, is equivalent to o-convergence\footnote{In order complete lattices, order bounded o-Cauchy nets are order convergent.}.\par

One can also easily see that, in $c_0$ or $\ell_p (1\leq p\leq \infty)$, uo-Cauchy of nets is the same as coordinate-wise Cauchy, and that, in $L_p(\mu) (1\leq p<\infty)$, uo-Cauchy of sequences is the same as almost everywhere Cauchy.

Note that every uo-convergent net is uo-Cauchy. In the rest of this subsection, we provide some sufficient conditions to yield uo-convergence of uo-Cauchy nets. The following are two obvious cases.

\begin{remark}\label{trivial}
\begin{enumerate}
\item\label{trivial1} If a uo-Cauchy net $(x_\alpha)$ has a uo-convergent subnet whose uo-limit is $x$, then $x_\alpha\xrightarrow{uo}x$. The proof is straightforward verification.
\item\label{trivial2} A norm convergent uo-Cauchy net uo-converges to its norm limit. One can easily show this by using the continuity of lattice operations with respect to norm.
\end{enumerate}
\end{remark}

The following are consequences of Propositions~\ref{ncon} and \ref{abs-con}.

\begin{proposition}\label{aobuc}In an order continuous Banach lattice, every almost order bounded uo-Cauchy net converges uo- and in norm to the same limit.\end{proposition}

\begin{proof}Suppose $(x_\alpha)$ is almost order bounded and uo-Cauchy. Then the net $(x_\alpha-x_{\alpha'})$ is almost order bounded and is uo-convergent to $0$. Thus, it converges to $0$ in norm by Proposition~\ref{ncon}. It follows that the net $(x_\alpha)$ is norm-Cauchy, and thus norm-convergent. Now apply Remark~\ref{trivial}\eqref{trivial2}.
\end{proof}

\begin{theorem}\label{wcuc}Let $X$ be order continuous. Then every relatively weakly compact uo-Cauchy net converges uo- and $\abs{\sigma}(X,X^*)$ to the same limit.
\end{theorem}

\begin{proof}Suppose $(x_\alpha)$ is relatively weakly compact and uo-Cauchy. Then it has a subnet $(x_\beta)$ weakly convergent to some $x\in X$. In view of Proposition~\ref{abs-con}, it suffices to prove $x_\alpha\xrightarrow{uo}x$.\par

Assume first that $X$ has a weak unit. Then $X^*$ has a strictly positive functional $x_0^*$ (cf.~\cite[Theorem 4.15]{ali}). Let $\widetilde{X}$ be the AL-space constructed for the pair $(X,x_0^*)$ as in Subsection~\ref{AL-rep}. By Lemma~\ref{al-wcon}, $(x_\alpha)$ is relatively weakly compact and thus almost order bounded in $\widetilde{X}$. By Lemma~\ref{uo-equiv}, $(x_\alpha)$ is also uo-Cauchy in $\widetilde{X}$. Therefore, $x_\alpha\xrightarrow[uo]{\norm{\cdot}_L}y$ for some $y\in \widetilde{X}$ by Proposition~\ref{aobuc}.
By Lemma~\ref{al-wcon} we have that $x_\beta\xrightarrow{{w}} x$ in $\widetilde{X}$, hence $y=x$. It follows that $x_\alpha\xrightarrow{uo}x$ in $\widetilde{X}$, and therefore, in $X$, by Lemma~\ref{uo-equiv} again.\par
The general case. Fix any $y>0$. Let $B$ be the band generated by $y$ in $X$ and $P$ the band projection onto $B$. Then $B$ is an order continuous Banach lattice with a weak unit. Clearly, $(Px_\alpha)$ is relatively weakly compact in $B$ and $Px_\beta\xrightarrow{w}Px$ in $B$. By Lemma~\ref{band-proj}, $(Px_\alpha)$ is uo-Cauchy in $B$. Thus the preceding case implies $Px_\alpha\xrightarrow{uo}Px$ in $B$. In particular, $\abs{x_\alpha-x}\wedge y=P(\abs{x_\alpha-x}\wedge y)=\abs{Px_\alpha-Px}\wedge y\xrightarrow{o}0$ in $B$, and therefore, in $X$ by the remark preceding Lemma~\ref{band-proj}.
\end{proof}

As in Remark~\ref{follow-ncon}, we can not expect norm convergence in this theorem.

\begin{remark}\label{c0exa} Note that, in general, we can not replace weak compactness with norm boundedness to obtain uo-convergence in this theorem. Indeed, let $X=c_0$, and put $x_n=\sum_1^ne_k$, where $(e_n)$ is the standard basis. Then $(x_n)$ is uo-Cauchy and norm bounded, but it is not uo-convergent in $c_0$.\end{remark}

In fact, the following theorem shows that $c_0$ is the only case that we need to avoid. We need the following sublattice version of Lemma~\ref{uo-ideal}.

\begin{lemma}\label{c0-case}Let $X$ be an order complete Banach lattice and $Y$ a closed sublattice of $X$. Suppose $Y$ is order continuous in its own right.
\begin{enumerate}
\item\label{c0-case1} For every order bounded above subset of $Y$, its supremum in $Y$ equals its supremum in $X$.
\item\label{c0-case2} For an order bounded net $(x_\alpha)\subset Y$, $x_\alpha\xrightarrow{o}0$ in $Y$ iff $x_\alpha\xrightarrow{o}0$ in $X$.
\item\label{c0-case3} For any net $(x_\alpha)\subset Y$, $x_\alpha\xrightarrow{uo}0$ in $Y$ iff $x_\alpha\xrightarrow{uo}0$ in $X$.
\end{enumerate}
\end{lemma}

\begin{proof}
For \eqref{c0-case1}, without loss of generality, assume $A=(x_\alpha)\subset Y$ with $x_\alpha\uparrow$. Let $x$ be its supremum in $Y$. Then by order continuity of $Y$, $x_\alpha\rightarrow x$ in norm of $Y$, and thus in norm of $X$. Since $x_\alpha$ is increasing, $x$ is its supremum in $X$.\par

\eqref{c0-case2} Suppose $x_\alpha\xrightarrow{o} 0$ in $Y$. Since $Y$ is order complete, the net $(y_\alpha)\subset Y$, where $y_\alpha:=\sup_{\beta\geq \alpha}\abs{x_\beta}$, is well defined and decreases to $0$ in $Y$. By \eqref{c0-case1}, $y_\alpha\downarrow0$ in $X$. Since $\abs{x_\alpha}\leq y_\alpha$, we have $x_\alpha\xrightarrow{o}0$ in $X$.
Conversely, suppose $x_\alpha\xrightarrow{o} 0$ in $X$. Since $X$ is order complete, the net $y_\alpha:=\sup_{\beta\geq \alpha}\abs{x_\beta}$ is well defined and decreases to $0$ in $X$. By \eqref{c0-case1}, $y_\alpha\in Y$. It is clear that $y_\alpha\downarrow 0$ in $Y$, thus $x_\alpha \xrightarrow{o}0$ in $Y$.\par

\eqref{c0-case3} Suppose $x_\alpha\xrightarrow{uo}0$ in $X$. Then for any $0<y\in Y$, $\abs{x_\alpha}\wedge y\xrightarrow{o}0$ in $X$, and thus in $Y$, by \eqref{c0-case2}. It follows that $x_\alpha\xrightarrow{uo}0$ in $Y$.
Conversely, suppose now $x_\alpha\xrightarrow{uo}0$ in $Y$. Let $I$ be the ideal generated by $Y$ in $X$. Take any $x\in I_+$. Then $x\leq y$ for some $y\in Y_+$. Note that $\abs{x_\alpha}\wedge y\xrightarrow{o}0$ in $Y$, and thus in $X$ by \eqref{c0-case2}, and in $I$ by the remark preceding Lemma~\ref{band-proj}. Therefore, $x_\alpha\xrightarrow{uo}0$ in $I$, and thus in $X$ by Lemma~\ref{uo-ideal}.
\end{proof}

\begin{remark}
It is enough to assume $X$ is $\sigma$-order complete in Lemma~\ref{c0-case}, if we only consider countable sets and nets with countable index sets.
\end{remark}

\begin{theorem}\label{cha} Let $X$ be an order continuous Banach lattice. The following are equivalent:
\begin{enumerate}
\item\label{cha1} $X$ is a KB-space;
\item\label{cha2} every norm bounded uo-Cauchy net in $X$ is uo-convergent;
\item\label{cha3} every norm bounded uo-Cauchy sequence in $X$ is uo-convergent.
\end{enumerate}
\end{theorem}

\begin{proof}
\eqref{cha2}$\Rightarrow$\eqref{cha3} is obvious. We now prove \eqref{cha3}$\Rightarrow$\eqref{cha1}. Assume \eqref{cha3} holds but $X$ is not a KB-space. Then $X$ contains a lattice copy of $c_0$. Without loss of generality, we assume $c_0\subset X$.
Recall from Remark~\ref{c0exa} that $x_n=\sum_1^ne_i$ is uo-Cauchy in $c_0$ but not uo-convergent in $c_0$ and it satisfies $\sup_n\norm{x_n}<\infty$. By Lemma~\ref{c0-case}\eqref{c0-case3}, $(x_n)$ is also uo-Cauchy in $X$, and thus uo-converges to some $x\in X_+$ by assumption \eqref{cha3}. Let $u$ be a weak unit of $c_0$, $B$ the band generated by $u$ in $X$ and $P_B$ the corresponding band projection. Then by Lemma ~\ref{band-proj}, we have that $x_n=P_Bx_n\xrightarrow{{uo}}P_Bx$ in $X$, thus $x=P_Bx$. Moreover, we have, for all $k\geq 1$, $|x_n\wedge ku-x\wedge ku|\leq |x_n-x|\wedge ku\rightarrow 0$ in order of $X$, hence $x_n\wedge ku$ converges to $x\wedge ku$ in norm of $X$ by order continuity of $X$. Since $(x_n\wedge ku)_n\subset c_0$, we have $x\wedge ku\in c_0$ for all $k\geq 1$, and therefore, $x=P_Bx=\lim_kx\wedge ku\in c_0$ by order continuity of $X$ again. Hence, by Lemma~\ref{c0-case}\eqref{c0-case3}, we have that $x_n$ uo-converges in $c_0$ to $x\in c_0$, a contradiction.\par

\eqref{cha1}$\Rightarrow$\eqref{cha2}. Suppose that $(x_\alpha)$ is uo-Cauchy and norm bounded. From $\abs{x^\pm-y^\pm}\leq \abs{x-y}$, we know  that $(x_\alpha^\pm)$ are both uo-Cauchy. Thus, without loss of generality, we assume that the net $(x_\alpha)$ consists of positive elements (cf.~Lemma~\ref{comparison}).\par

Assume first that $X$ has a weak unit $x_0$. Fix $k \in \mathbb{N}$.
Note that $|x_\alpha\wedge k x_0-x_{\alpha'}\wedge kx_0|\leq |x_\alpha-x_{\alpha'}|\wedge kx_0$. Hence, $(x_\alpha\wedge kx_0)$ is order Cauchy. Since it is order bounded, it converges in order (and thus in norm) to some $y_k\in X$. It is clear that $\sup_k||y_k||\leq \sup_{k}\sup_\alpha\norm{x_\alpha\wedge kx_0}\leq \sup_\alpha\norm{x_\alpha}<\infty$. Note also that $\{y_k\}_k$ is increasing. Thus $y_k$ converges to some $y\in X$.\par

It remains to prove $x_\alpha\xrightarrow{uo}y$, or equivalently, $|x_\alpha-y|\wedge x_0\xrightarrow{o}0$ by Lemma~\ref{uoc}.
Put $x_{\alpha,\alpha'}=\sup_{\beta\geq \alpha,\beta'\geq \alpha'}|x_\beta-x_{\beta'}|\wedge x_0$. Then by assumption, $x_{\alpha,\alpha'}\downarrow 0$. Now for any $k\geq 1$, we have
$$|x_\beta\wedge kx_0-x_{\beta'}\wedge kx_0|\wedge x_0\leq |x_\beta-x_{\beta'}|\wedge x_0\leq x_{\alpha,\alpha'},\;\;\forall\beta\geq \alpha,\beta'\geq\alpha'.$$
Taking limit in $\beta'$, we have, by the continuity of lattice operations with respect to norm, for any $\beta\geq \alpha$,
$$|x_\beta\wedge kx_0-y_k|\wedge x_0\leq x_{\alpha,\alpha'},\;\;\forall\;k\geq 1.$$
Now letting $k\rightarrow \infty$ and using continuity of lattice operations again, we have
$$|x_\beta-y|\wedge x_0\leq x_{\alpha,\alpha'},\;\;\forall\;\beta\geq \alpha,$$
from which it follows that $|x_\alpha-y|\wedge x_0\xrightarrow{o}0$.\par

The general case. Let $B_\delta$ and $P_\delta$ be as in the proof of Theorem~\ref{spsp}. Then each $B_\delta$ is a KB-space with a weak unit. Since $X$ is order continuous, we have that, for each $x \in X$, $P_\delta x \rightarrow x$ in norm. By Lemma~\ref{band-proj}, $(P_\delta x_\alpha)$ is uo-Cauchy in $B_\delta$. The preceding case implies that there exists $0\leq z_\delta\in B_\delta$ such that $P_\delta x_\alpha\xrightarrow{uo}z_\delta$ in $B_\delta$, and hence in $X$, by Lemma~\ref{uo-ideal}.
Note that the net $(z_\delta)$ is increasing by Lemma~\ref{comparison} and $\sup_\delta\norm{z_\delta}\leq \sup_\alpha\norm{x_\alpha}<\infty$ by Lemma~\ref{norm}. Thus it converges to some $0\leq x\in X$.\par

It remains to show that $x_\alpha\xrightarrow{uo}x $. Pick any $y\in X_+$. Let $P_y$ be the band projection onto $B_y$. Then a similar argument as before shows that $P_yx_\alpha$ uo-converges to some $0\leq y_0\in B_{y}$ in $X$. We have

\begin{equation}\label{eq00}
\abs{x_\alpha-y_0}\wedge y=P_y(\abs{x_\alpha-y_0}\wedge y)=\abs{P_yx_\alpha-y_0}\wedge y\xrightarrow{o}0 \text{ in  }X.
\end{equation}
Thus for any $\delta$, $\bigabs{P_\delta x_\alpha-P_\delta y_0}\wedge y\leq |x_\alpha-y_0|\wedge y\xrightarrow{o}0$ in $X$.
Recall that $P_\delta x_\alpha\xrightarrow{uo}z_\delta$ in $X$. Thus, $\bigabs{P_\delta x_\alpha-z_\delta}\wedge y\xrightarrow{o}0$ in $X$. We have \begin{align}\label{mid-cite}\abs{z_\delta-P_\delta y_0}\wedge y=0.\end{align}Since $\lim P_\delta y_0=y_0$ and $\lim z_\delta=x$, taking limit in $\delta$ in \eqref{mid-cite}, we have $0=\abs{x-y_0}\wedge y=\bigabs{P_yx-y_0}\wedge y$. It follows that $y_0=P_yx$ and $\abs{x_\alpha-x}\wedge y=\bigabs{P_yx_\alpha -P_yx}\wedge y=\bigabs{P_yx_\alpha -y_0}\wedge y\xrightarrow{o}0$ in $X$ by (\ref{eq00}).
\end{proof}

In the classical $L_1$ case, Fatou's lemma says that, if $f_n$ converges almost everywhere to a measurable function $f$, then $\norm{f}_1\leq\liminf_n\norm{f_n}$. In particular, if $\{f_n\}$ is bounded, then $f\in L_1$. In view of this, the preceding theorem may be viewed as \textbf{Abstract Fatou Lemma} (cf.~Lemma~\ref{norm}).\par

\begin{remark}
We can not replace order continuity of $X$ with order completeness in the assumption of this theorem. Indeed, it is easily seen that in $\ell_\infty$, every norm bounded uo-Cauchy net is also order Cauchy, and thus is order convergent, since $\ell_\infty$ is order complete. Nevertheless, $\ell_\infty$ is not a KB-space.
\end{remark}

\begin{remark}
Suppose that $X$ is an order continuous Banach lattice which is realized as a K\"{o}the space within some $L_1(\mu)$ space. Theorem \ref{cha} says, in particular, that $X$
is a KB-space iff given any bounded sequence $(x_n)$ in $X$ which converges
a.e.~to some measurable function x, then $x \in X$. This gives a nice
characterization of KB-spaces among order continuous K\"{o}the spaces
in terms of almost everywhere convergence.
\end{remark}

We mention here that many results in this section and the previous one can be easily extended to \textbf{unbounded norm convergence}. We say that a net $(x_\alpha)$ is unbounded norm convergent to $x$ if $|x_\alpha-x| \wedge y \xrightarrow{{||\cdot||}} 0$ for all $y \in X_+$. This type of convergence is introduced in \cite{tro04} and is an analogue of convergence in measure in $L_1(\mu)$ spaces.

\section{Convergence of Submartingales}\label{martingale}

The classical Doob's (sub-)martingale convergence theorem in a probability $L_1(\mu)$-space is as follows (cf.~\cite[Theorem~9.4.4]{chu74}).
\begin{theorem}[Doob]\label{doob}
Every norm bounded submartingale in $L_1(\mu)$ converges almost surely.
\end{theorem}

In this section, we extend it to a measure-free setting. We refer the reader to \cite{chu74,meyer} for Doob's classical theorems and to \cite{dem66,tro11,koro08,kuo04,kuo05,kuo06,sto,tro05,uhl71} for various efforts of measure-free extensions in the history.

Let's recall some basic definitions. Let $X$ be a vector lattice. A \term{filtration} $(E_n)$ on $X$ is a sequence of positive projections on $X$ such that $E_nE_m=E_mE_n=E_{m\wedge n}$ for all $m,n\geq 1$. We say that the filtration $(E_n)$ is \term{bounded} if $X$ is a Banach lattice and $\sup_{n}\norm{E_n}<\infty$. Recall also that a sequence $(x_n)\subset X$ is called a \term{martingale} relative to the filtration $(E_n)$ if $E_nx_m=x_n$ for all $m\geq n$, and a sequence $(z_n)\subset X$ is called a \term{submartingale} relative to $(E_n)$ if $z_n\in\mathrm{Range}(E_n)$ and $E_nz_m\geq z_n$ for all $m\geq n$.

\subsection{Abstract bistochastic filtrations}
In view of \cite[Definition 5.49]{abr}, we say that a filtration $(E_n)$ on a vector lattice $X$  is \textbf{abstract bistochastic} if there exist a weak unit $x_0>0$ in $X$ and a strictly positive order continuous functional $x_0^*>0$ on $X$ such that the following double condition is satisfied by $E_1$ (and hence by all $E_n$'s):
$$E_1x_0=x_0;\quad E_1^*x_0^*=x_0^*.\leqno(\diamond) $$
To our knowledge, these conditions were first considered in the setting of abstract martingales by \cite{kuo05} under the terminology of expectation operators (see~\cite[Definition~6.2 and Theorem~6.4]{kuo05}).

\begin{remark}\label{martin99}Note that the classical filtrations on Banach lattice valued $L_1(\Omega;F)$-spaces are abstract bistochastic. Details are as follows. We refer the reader to \cite{DU77,Egg84} for unexplained terminology.\par

Let $(\Omega,\mathscr{F},P)$ be a probability space and $F$ a Banach lattice. Then $X:=L_1(\Omega,F)$ is a Banach lattice. Moreover, $X$ is an order continuous Banach lattice (respectively, a KB-space) if so is $F$.
For our purpose, \textbf{we always assume $F$ is order continuous and has a weak unit $x_0$}. Then $F$ has a strictly positive functional $x^*_0>0$. Define the constant function $f_0(\omega)=x_0$ and the constant functional $g_0(w)=x_0^*$ for each $\omega \in \Omega$. It is straightforward to verify that $f_0$ is a weak unit in $X$ and $g_0$ is a strictly positive functional on $X$.\par

Let $\mathscr{G}$ be a sub-$\sigma$-field of $\mathscr{F}$ and $E:=E(\cdot|\mathscr{G})$ the classical conditional expectation defined on $X$ relative to $\mathscr{G}$. Then $E$ is a positive projection on $X$. Since $E$ preserves the constant functions(cf.~\cite[Proposition I.2.2.3]{Egg84}), we have $Ef_0=f_0$. Moreover, for each $x \in X$, we have
\begin{align*}
(E^*g_0)(x)=&g_0(Ex)=\int x_0^*\Big((Ex)(\omega)\Big)\mathrm{d}\omega\\
=&x_0^*\left(\int (Ex)(\omega)\mathrm{d}\omega\right)=x_0^*\left(\int x(\omega)\mathrm{d}\omega\right)=g_0(x).
\end{align*}
It follows that $E^*g_0=g_0$.
\end{remark}

Let's also observe the following well-known simple facts concerning the condition $(\diamond)$. We provide the proofs for the sake of completeness.

\begin{lemma}\label{double1} Let $X$ be a vector lattice with a strictly positive order continuous functional $x_0^*>0$ and $E$ a positive projection on $X$. The following are equivalent:
\begin{enumerate}
\item\label{double11} $E^*x^*=x^*$ for some strictly positive order continuous functional $x^*>0$;
\item\label{double12} $E$ is strictly positive and order continuous.
\end{enumerate}
\end{lemma}

\begin{proof}\eqref{double11}$\Rightarrow$\eqref{double12} Suppose $Ex=0$ for $x\geq 0$. Then $x^*(x)=(E^*x^*)(x)=x^*(Ex)=0$. Thus, $x=0$ due to strict positivity of $x^*$. This proves $E$ is strictly positive. Now suppose $x_\alpha\downarrow 0$ and $Ex_\alpha\geq z\geq 0$ for all $\alpha$. Then $x^*(z)\leq x^*(Ex_\alpha)=x^*(x_\alpha)\downarrow0$ by order continuity of $x^*$. It follows that $x^*(z)=0$, and thus $z=0$. Hence, $Ex_\alpha\downarrow 0$, and $E$ is order continuous.\par

\eqref{double12}$\Rightarrow$\eqref{double11} Simply put $x^*=E^*x^*_0$. It is straightforward to verify that $x^*$ is strictly positive and order continuous and $E^*x^*=x^*$.
\end{proof}

\begin{lemma}\label{double2}Let $X$ be a normed lattice with a quasi-interior point $x_0>0$ and $E$ a positive projection on $X$. The following are equivalent:
\begin{enumerate}
\item\label{double21} $Ex=x$ for some quasi-interior point $x>0$;
\item\label{double22} $E^*$ is strictly positive.
\end{enumerate}
\end{lemma}

\begin{proof}\eqref{double21}$\Rightarrow$\eqref{double22} Suppose $E^*x^*=0$ for some $x^*\geq 0$. Then $x^*(x)=x^*(Ex)=(E^*x^*)(x)=0$. Since $x$ is a quasi-interior point, $x^*=0$. Thus, $E^*$ is strictly positive.
\eqref{double22}$\Rightarrow$\eqref{double21} Put $x=Ex_0$. Clearly, $Ex=x$. It remains to prove $x$ is a quasi-interior point. Indeed, take $x^*>0$, then $E^*x^*>0$. Thus $x^*(x)=x^*(Ex_0)=E^*x^*(x_0)>0$. It follows that $x$ is a quasi-interior point by \cite[Theorem~4.85]{ali}.
\end{proof}

\begin{corollary}\label{ocdouble}Let $X$ be an order continuous Banach lattice with a weak unit and $(E_n)$ a filtration on $X$. Then $(E_n)$ is abstract bistochastic iff $E_1$ and $E_1^*$ are both strictly positive iff $E_n$ and $E_n^*$ are both strictly positive for every $n$.\end{corollary}

\begin{proof}
Observe that weak units in $X$ are quasi-interior points, that $X$ has a strictly positive functional (cf.~\cite[Theorem~4.15]{ali}), and that all positive operators on $X$ are order continuous. Apply Lemmas~\ref{double1} and \ref{double2}.
\end{proof}

\subsection{Uo-convergence of submartingales}
The following is a version of Doob's almost sure convergence theorem in vector lattices.

\begin{theorem}\label{pcon1} Let $X$ be an order complete vector lattice and $(E_n) $ an abstract bistochastic filtration on $X$. Let $(z_n)$ be a submartingale relative to $(E_n)$. If $\sup_nx_0^*(z_n^+)<\infty$ where $x_0^*$ is as in the double condition $(\diamond)$, then $(z_n)$ is uo-Cauchy. In particular, if $(z_n)$ is order bounded, then $(z_n)$ is order convergent.\end{theorem}

\begin{proof}Let $x_0$ be also as in the double condition $(\diamond)$. Without loss of generality, assume $x_0^*(x_0)=1$.
Let $\widetilde{X}$ be the AL-space constructed for the pair $(X,x_0^*)$ as in Subsection~\ref{AL-rep}. Then by Lemma~\ref{representation}, $\widetilde{X}=L_1(P)$ for some probability $P$ with $x_0$ corresponding to the constant $1$ function.\par

By Lemma~\ref{ope}, each $E_n$ extends to a contractive positive operator $\widetilde{E_n}$ on $\widetilde{X}$. Clearly, each $\widetilde{E_n}$ is still a projection. Since $\widetilde{E_n}1=1$, we know, by Douglas' representation theorem(\cite[Corollary~5.52]{abr}), that each $\widetilde{E_n}$ is a conditional expectation on $L_1(P)$. It follows that $\{\widetilde{E_n}\}$ is a classical filtration on $L_1(P)$.\par

For any $n\leq m$, we have $x_0^*(z_n)\leq x_0^*(E_nz_m)=x_0^*(z_m)\leq x_0^*(z_m^+)$. Hence, for all $k$, we have $x_0^*(z_1)\leq x_0^*(z_k)\leq \sup_nx_0^*(z_n^+)$. It follows that $\sup_n\abs{x_0^*(z_n)}<\infty$, and therefore, $\sup_n x_0^*(\abs{z_n})<\infty$, i.e.~$\sup_n\norm{z_n}_L<\infty$. We know, by Doob's Theorem~\ref{doob}, that $(z_n)$ converges almost everywhere (to a function in $\widetilde{X}=L_1(P)$). Therefore, $(z_n) $ is uo-Cauchy in $\widetilde{X}$, and hence in $X$ by Lemma~\ref{uo-equiv}.
\end{proof}

According to \cite[Theorem 6]{cha64}, every norm bounded martingale in $L_1(\Omega;F)$ with respect to classical filtrations converges strongly almost surely for every probability space $\Omega$ iff $F$ has the Radon-Nikodym property. Below we prove that even when $F$ lacks the Radon-Nikodym property, norm bounded submartingales in $L_1(\Omega;F)$ still have certain convergence property.

\begin{corollary}\label{vmart}
Let $F$ be an order continuous Banach lattice with a weak unit. Then every norm bounded submartingale $(z_n)$ in $L_1(\Omega;F)$ with respect to a classical filtration is almost surely uo-Cauchy in $F$ (i.e.~outside a subset of measure $0$, $z_n(\omega)$ is uo-Cauchy in $F$).
\end{corollary}

\begin{proof}Let $x_0$ and $f_0$ be as in Remark~\ref{martin99}.  By Remark~\ref{martin99} and Theorem~\ref{pcon1}, we know that $(z_n)$ is uo-Cauchy in $L_1(\Omega;F)$. Now observe that $(z_n)$ is uo-Cauchy in $L_1(\Omega;F)$ iff (by Lemma~\ref{uoc}) $\abs{z_n-z_m}\wedge f_0\xrightarrow{o}0$ in $L_1(\Omega;F)$, iff $\sup_{n'\geq n,m'\geq m}\abs{z_{n'}-z_{m'}}\wedge f_0\downarrow 0$ in $L_1(\Omega;F)$, iff outside a subset of measure $0$, $\sup_{n'\geq n,m'\geq m}\abs{z_{n'}(\omega)-z_{m'}(\omega)}\wedge x_0\downarrow 0$ in $F$, iff outside a subset of measure $0$, $(z_n(\omega))$ is uo-Cauchy in $F$.
\end{proof}

Note that, in general, we can not expect uo-convergence in Theorem~\ref{pcon1}.

\begin{example}[{\cite[Example~6]{tro11}}]\label{vlad_ex}Let $X=c_0$ and put $E_1$ to be the projection on $c_0$ consisting of $2\times 2$ diagonal blocks \begin{math}
  \Bigl[
  \begin{smallmatrix}
    \frac{1}{2} &\frac{1}{2}\\
 \frac{1}{2}     & \frac{1}{2}
  \end{smallmatrix}
  \Bigr]
\end{math}. Define $E_n$ by replacing the first $(n-1)$ diagonal blocks in $E_1$ with the identity matrix. Then $(E_n)$ is a (bounded) filtration on $c_0$. Since both $E_1$ and $E_1^*$ are strictly positive, $(E_n)$ is abstract bistochastic, by Corollary~\ref{ocdouble}. For each $n\geq 1$, put $$x_n=\big(\overbrace{1,-1,1,-1,\dots,1,-1}^{2n-2\mbox{ coordinates}},0,0,\dots\big).$$
Then $(x_n)$ is a norm bounded martingale relative to $(E_n)$. It is clear that $(x_n)$ is uo-Cauchy but not uo-convergent in $c_0$.
\end{example}

The following Proposition~\ref{nconf} establishes uo-convergence of submartingales. For the proof of the proposition, we need the following critical observation which has appeared in \cite{uhl71} and \cite{tro05}.

\begin{lemma}\label{weaksub}Let $(E_n)$ be a filtration on a Banach lattice.
\begin{enumerate}
\item\label{weaksub1} If $(x_n)$ is a martingale such that a subsequence $(x_{n_k})$ weakly converges to $x$, then $x_n=E_nx$ for all $n$.
\item\label{weaksub2} If $(z_n)$ is a submartingale such that a subsequence $(z_{n_k})$ weakly converges to $x$, then $z_n\leq E_nx$ for all $n$.
\end{enumerate}
\end{lemma}

\begin{proof}We prove \eqref{weaksub2} only. Fix $n$. Then for sufficiently large $k$, we have $z_n\leq E_nz_{n_{k}}$. Since $E_nz_{n_k}\xrightarrow{w}{E_nx}$, the desired result follows.
\end{proof}

\begin{proposition}\label{nconf}Let $X$ be an order continuous Banach lattice and $(E_n) $ an abstract bistochastic filtration on $X$. Let $(z_n)$ be a submartigale relative to $(E_n)$.
\begin{enumerate}
\item\label{nconf1} If $X$ is a KB-space and $\sup_n\norm{z_n}<\infty$, then $(z_n)$ is uo-convergent.
\item\label{nconf2} If a subsequence of $(z_n)$ converges weakly to $x$, then $z_n \xrightarrow{{uo}} x$.
\end{enumerate}
\end{proposition}

\begin{proof}
\eqref{nconf1} follows from Theorem~\ref{pcon1} and Theorem~\ref{cha}. For \eqref{nconf2}, suppose $z_{n_k}\xrightarrow{w} x$. Let $x_0^*$ be as in the double condition $(\diamond)$. Since $z_n\leq E_nx$ by Lemma~\ref{weaksub}, $z_n^+\leq E_nx^+$. Thus $\sup_nx_0^*(z_n^+)\leq \sup_nx_0^*(E_nx^+)=x_0^*(x^+)<\infty$. It follows from Theorem~\ref{pcon1} that $(z_n)$ is uo-Cauchy. In particular, $(z_{n_k})$ is also uo-Cauchy. Since $ (z_{n_k})$ is also relatively weakly compact, it uo-converges to $x$ by Theorem~\ref{wcuc}. Therefore, $(z_n)$ uo-converges to $x$, by Remark~\ref{trivial}\eqref{trivial1}.
\end{proof}

\subsection{Norm convergence of submartingales}
In the classical case, we have the following norm convergence theorems for (sub-)martingales(cf. ~\cite[Chapter 9]{chu74}, \cite[Chapter II]{meyer}).

\begin{theorem}\label{n-con}
\begin{enumerate}
\item\label{n-con1} In $L_1(\mu)$, if $(z_n)$ is a relatively weakly compact submartingale then $(z_n)$ converges in norm. %~\cite{chu74} Theorem 9.4.5
\item\label{n-con2} In $L_p(\mu)(1 \leq p < \infty)$, every relatively weakly compact martingale converges in norm. %~\cite{meyer} Chapter II, Theorems 13,15
\item\label{n-con3} In $L_p(\mu)(1 < p < \infty)$, if $(z_n)$ is a norm bounded submartingale then $(z_n^+)$ converges in norm.
\end{enumerate}
\end{theorem}
We now extend these results to abstract (sub-)martingales. The following is immediate by Proposition~\ref{aobuc}, Theorems~\ref{wcuc} and \ref{pcon1}, and extends Theorem~\ref{n-con}\eqref{n-con1}.

\begin{proposition}\label{helpful}Let $X$ be an order continuous Banach lattice and $(E_n) $ an abstract bistochastic filtration on $X$. Let $(z_n)$ be a submartigale relative to $(E_n)$.
\begin{enumerate}
\item If $(z_n)$ is almost order bounded, then $(z_n)$ converges uo- and in norm to the same limit.
\item If $(z_n)$ is relatively weakly compact, then $(z_n)$ converges uo- and $\abs{\sigma}(X,X^*)$ to the same limit.
\end{enumerate}
\end{proposition}

The next result extends Theorem~\ref{n-con}\eqref{n-con2} and \eqref{n-con3}.

\begin{theorem}\label{partial}
Let $X$ be an order continuous Banach lattice and $(E_n)$ a bounded abstract bistochastic filtration on $X$.
\begin{enumerate}
\item\label{partial1} For any martingale $(x_n)$ relative to $(E_n)$, if a subsequence $(x_{n_k})$ weakly converges to $x$, then $(x_n)$ converges uo- and in norm to $x$;
\item\label{partial2} For any submartingale $(z_n)$ relative to $(E_n)$, if a subsequence $(z_{n_k})$ weakly converges to $x$, then $(z_n^+)$ converges uo- and in norm to $x^+$.
\end{enumerate}
\end{theorem}

\begin{proof}
For \eqref{partial1}, suppose $x_{n_k}\xrightarrow{w} x$. By Proposition~\ref{nconf}, we know $x_n$ uo-converges to $x$. By Lemma~\ref{weaksub}, we have $x_n=E_nx$ for $n\geq 1$. Put $M=\bigcup_{n=1}^\infty\mathrm{ Range}(E_n)$. Then $x \in\overline{M}^{w}= \overline{M}^{\norm{\cdot}}$, and thus $x_n=E_nx\rightarrow x$ (this is true for all $y\in M$, and thus is also true for $x\in \overline{M}^{\norm{\cdot}} $ due to the boundedness of $(E_n)$).\par

For \eqref{partial2}, we apply Proposition~\ref{helpful}. Let $z_{n_k}\xrightarrow{w} x$, then $x \in \overline{M}^{\norm{\cdot}}$. Since the range of each $E_n$ is a sublattice of $X$ (\cite[Theorem~5.59(iv)]{abr}), so is $\overline{M}^{\norm{\cdot}}$. It follows that $x^+\in\overline{ M}^{\norm{\cdot}}$. Thus $(E_nx^+)$ is norm convergent, and in particular, is almost order bounded. Since $z_n\leq E_nx$ by Lemma~\ref{weaksub}, we have $z_n^+\leq E_nx^+$ for all $n$, implying that $(z_n^+)$ is also almost order bounded. Moreover, it is straightforward to verify that $(z_n^+)$ is a submartingale relative to $(E_n)$. Therefore, $(z_n^+)$ converges uo- and in norm to the same limit, by Proposition~\ref{helpful}. By Proposition~\ref{nconf}, $z_n$ uo-converges to $x$. Hence, the uo-limit of $z_n^+$ is $x^+$.
\end{proof}

\begin{remark}The norm convergence in Theorem~\ref{partial}\eqref{partial1} and its proof are due to Uhl \cite[Corollary 3]{uhl71} and \cite[Theorem~17]{tro05}. We give here a direct proof of the norm convergence in Theorem~\ref{partial}\eqref{partial2}. It follows from Lemma~\ref{weaksub} that $z_n \leq E_nx$ for each $n \in \mathbb{N}$. Thus $(z_n^+-x^+)^+\leq (z_n-x)^+\leq (E_nx-x)^+\xrightarrow{\norm{\cdot}}0$, since $x\in \overline{M}^{\norm{\cdot}}$. Therefore, $$ z_n^+ \vee x^+=(z_n^+-x^+)^+ +x^+ \xrightarrow{||\cdot||} x^+.$$Since $z_n \xrightarrow{{uo}} x$ by Proposition~\ref{nconf} , we have that $z_n^+\xrightarrow{uo}x^+$ and $z_n^+ \wedge x^+ \xrightarrow{{uo}} x^+$. Therefore, $z_n^+ \wedge x^+ \xrightarrow{o} x^+$. It follows that $$ z_n^+ \wedge x^+ \xrightarrow{{|| \cdot ||}} x^+.$$ Combining the above, we get $|z_n^+-x^+|=z_n^+ \vee x^+- z_n^+ \wedge x^+ \xrightarrow{||\cdot||}0$.
\end{remark}

We have used the fact that if $x\in\overline{\cup_1^\infty\mathrm{Range}(E_n)}$ then $(E_nx)$ converges in norm. This is actually true for arbitrary $x\in X$ if the filtration is bounded and abstract bistochastic; see~\cite[Chapter II, Theorem~13]{meyer} for the classical case.

\begin{theorem}\label{acon1} Let $X$ be an order continuous Banach lattice and $(E_n)$ a bounded abstract bistochastic filtration. Then for any $x\in X$, $(E_nx)$ converges uo- and in norm to the same limit.
\end{theorem}

\begin{proof} Let $x_0$ be as in the double condition $(\diamond)$. Put $C:=\sup_n\norm{E_n}$. Without loss of generality, assume $x>0$. For any $\varepsilon>0$, take $k_0$ such that $\norm{x-x\wedge k_0x_0}<\varepsilon$. Then \begin{align}\label{a1}\norm{E_n(x\wedge k_0x_0)-E_nx}\leq C\varepsilon,\quad \forall\;n\geq 1.\end{align} Since $0\leq E_n(x\wedge k_0x_0)\leq E_n(k_0x_0)=k_0x_0$, we have $\{E_nx\}\subset [0,k_0x_0]+C\varepsilon B_X$. It follows that $(E_nx)$ is almost order bounded. Thus it converges uo- and in norm to the same limit by Proposition~\ref{helpful}.
\end{proof}

\medskip

\textbf{Acknowledgements}. The authors would like to thank Dr.~Vladimir Troitsky for suggesting them to study abstract martingales and for many helpful discussions.

\end{document}